\documentclass[11pt]{amsproc}
\usepackage{mathrsfs}
\usepackage{stmaryrd}
\usepackage{cases}
\usepackage{amssymb}
\usepackage{amsmath}
\usepackage{amsfonts}
\usepackage{graphicx}
\usepackage{amsmath,amstext,amsbsy,amssymb}
\usepackage{color}
\newtheorem{theorem}{Theorem}[section]

\newtheorem{lem}[theorem]{Lemma}
\newtheorem{prop}[theorem]{Proposition}
\newtheorem{cor}[theorem]{Corollary}

\newtheorem{example}[theorem]{Example}

\theoremstyle{remark}

\numberwithin{equation}{section} \errorcontextlines=0

\newcommand{\la}{\lambda}

\begin{document}

\title[Spin characters of generalized symmetric groups]
{Spin characters of generalized symmetric groups}
\author{Xiaoli Hu}
\address{Hu: School of Sciences,
South China University of Technology, Guangzhou 510640, China}
\email{xiaolihumath@163.com}
\author{Naihuan Jing}
\address{Jing: Department of Mathematics,
   North Carolina State University,
   Raleigh, NC 27695, USA
   }
\email{jing@math.ncsu.edu}
\thanks{*Corresponding author: Naihuan Jing}

\begin{abstract}
In 1911 Schur computed the spin character values of the symmetric group using two important
ingredients: the first one later became famously known as the Schur Q-functions
and the second one was certain creative construction of the projective
characters on Clifford algebras.
In the context of the McKay correspondence and affine Lie algebras, the first part was
generalized to all wreath products by the vertex operator calculus in \cite{FJW} where a large part of the character table
was produced. The current paper generalizes the second part and provides the missing projective character values
for the wreath product of the symmetric group with a finite abelian group. Our approach relies on Mackey-Wigner's
little groups to construct irreducible modules. In particular, projective modules and spin character values of
all classical Weyl groups are obtained.
\end{abstract}
\maketitle

\section{Introduction}\label{intro}

The spin group $\widetilde{S}_n$ is a double cover of the symmetric
group $S_n$. 
In the seminal paper \cite{S} Schur generalized Frobenius theory and determined
all irreducible projective characters
of the symmetric group $S_n$ by introducing a new family of
symmetric functions later known as Schur Q-functions. These
symmetric functions play the same role for the spin group $\widetilde{S}_n$ as Schur functions do for
the symmetric group $S_n$. Schur showed further that though
the projective character values are for the most part given by Schur Q-functions, a significant portion
was provided by special spin modules and Clifford algebras.

After the classical work of Frobenius and Schur, irreducible characters of the general wreath products
${\Gamma}_n=\Gamma\wr S_n$ were constructed by Specht in his dissertation \cite{Sp}. The generalized symmetric groups
were also studied by Osima in \cite{O}, and Zelevinsky \cite{Ze} investigated the Hopf algebra structure of the
Grothendieck groups for all ${\Gamma}_n$.

During the last several decades there has been a resurgence of activities on the spin group.
Stembridge \cite{St} gave a combinatorial definition of Schur Q-functions, Sergeev
\cite{Ser} found that the hyperoctahedral group of the symmetric group has a similar character
theory, J\'ozefiak \cite{Jo} gave a modern account of Schur's work using superalgebras,
Nazarov \cite{Na} constructed all irreducible representations of the spin group,
Hoffman and Humphreys \cite{HH} also used Zelevinsky's method to study the
double covering groups  $\widetilde{\Gamma}_n$, and the second author \cite{J} provided a vertex operator approach to Schur Q-functions as well as projective character values.
Breakthroughs were also made on modular projective representations of the symmetric groups \cite{BK} \cite{K} (see also \cite{Be}).

On the other hand, recognizing the deep connection with the McKay correspondence, I. Frenkel, Jing and Wang \cite{FJW}
generalized the first part of Schur's work and determined all irreducible characters of
the spin wreath product $\widetilde{\Gamma}_n$
of a finite group $\Gamma$
and the symmetric group $S_n$.
When $\Gamma$ is a finite cyclic group, they are double covering groups of the generalized
symmetric groups, which include hyperoctahedral groups as special cases
when $\Gamma$ is of order 2.
In Schur's original work on $\widetilde{S}_n$, the projective characters of $S_n$ are parameterized by strict partitions. Again in the wreath products, the projective representations are in one to one correspondence to strict partition valued
functions or strict colored partitions.

It is well-known that the
projective character table of $\Gamma_n$ consists of the character values
on the so-called split conjugacy classes which can be divided
into two subsets: the {\it even} conjugacy classes corresponding to partition valued functions with
odd integer parts and the {\it odd} conjugacy classes corresponding to odd partition
valued functions with distinct parts.
In \cite{FJW} the authors determined all irreducible characters of spin wreath products by vertex operator calculus
and also showed that the character values at all odd conjugacy classes
are given by matrix coefficients of products of twisted vertex operators, thus solved
a big chunk of the character table.
It seems that the character values on odd strict colored partitions are beyond the reach of vertex operators. Later in \cite{AM} \cite{MJ} spin characters for
generalized symmetric groups were also considered using combinatorial methods and
certain basic spin character values were computed. However the character values on
odd strict partition valued functions are still unknown, as the method associated with the McKay correspondence and vertex representations seems
not suitable for computing this part of the character table. Knowledge of this
will be useful in representation theory
as they include practically all double coverings of Weyl groups of classical types.

Spin character values have been studied in the physics literature as well. In \cite{RW1} it was observed that {plethysms} play an important role in determining characters for spin characters of  SO$(n, \mathbb C)$ and
the spin group $\widetilde{S}_n$. This was later generalized to spin groups associated to orthogonal and symplectic Weyl groups \cite{RW2} and new algorithms were developed for computing the spin character values of Weyl groups.

The purpose of this paper is to obtain the missing part of the
character table of spin wreath products $\widetilde{\Gamma}_n$ for
the cases of an abelian group $\Gamma$. We construct all irreducible
characters by certain induced representations of Young subgroups of
$\widetilde{\Gamma}_n$ using the Mackey-Wigner method of little groups
(cf. \cite{Serre}).
Then we compute the
 spin character tables of the wreath products
$\widetilde{\Gamma}_n$.
In particular this includes, in principle, the
irreducible spin character values of Weyl groups of all classical types.

In the viewpoint of the new form of the McKay correspondence \cite{FJW} the problem of determining all
spin wreath products 
$\mathbb Z_{r+1}\wr \widetilde{S}_n$
amounts to a realization of the twisted affine Lie algebra
$A_r^{(1)}[-1]$. On the other hand,  Ariki \cite{A} has shown that the Grothendieck group of the category of
modules for the cyclotomic Hecke algebra $H_q(\mathbb Z_{r+1}\wr S_n)$ realizes the dual canonical basis
for the quantum affine Lie algebra $U_q(A_r^{(1)})$, which in turn gives
the decomposition matrix for the modular representations of the symmetric groups
by the Lascoux-Leclerc-Thibon algorithm
\cite{LLT}. These are partly the reasons that we study spin representations of
generalized symmetric groups in this paper, besides historic interest.

The paper is organized as follows. In the first two sections we discuss the basic notions
of the wreath products and the Grothendieck group of projective representations of the
wreath products. The twisted products of two spin modules are thoroughly reviewed and
special attention is paid to the case of cyclic groups. In section three we first
recall the basic spin representations and then use the Mackey-Wigner method of little groups
to decompose the orbits of Young subgroups. We construct all spin irreducible representations
indexed by strict partition valued functions, and then we show that the character values
are sparsely zero and the non-zero values are
given according to how the partitions are supported on various conjugacy classes.

\section{The spin  wreath products $\widetilde{\Gamma}_n$ .}\label{sec:1}
\subsection{The spin group $\widetilde{S}_n$.}
The spin group $\widetilde{S}_n$ is the finite group generated by
$z$ and $t_i, ~(i=1, \cdots, n-1)$ with the defining relations:
\begin{align}\label{eq:def4spin}
& z^2=1, \ \ t_i^2=(t_it_{i+1})^3=z,\\
& t_it_j=zt_jt_i, \ \ |i-j|>1,\\
& zt_i=t_iz.
\end{align}
The group $\widetilde{S}_n$ is a central extension of $S_n$ by the cyclic
group $\mathbb{Z}_2$, as the map $\theta_n$ sending
$t_i$ to the transposition $(i,i+1)$ and $z$ to 1
is a homomorphism from $\widetilde{S}_n$ to $S_n$. In fact
Schur \cite{S} has shown that the spin group $\widetilde{S}_n$ is
one of the two non-trivial double covers of the symmetric group $S_n$ ($n\geq 4$ but $n\neq 6$).

We recall the Conway cycle presentation for $\widetilde{S}_n$ \cite{Ws}.
For each $k\in \{1, \cdots , n\}$, let $x_k=t_kt_{k+1}\cdots t_n
\cdots t_{k+1}t_k\in \widetilde{S}_{n+1}$ . For distinct integers $ i_1,
\cdots, i_m \in \{1, 2, \cdots,
n\}$, we define the cycle $[i_{1}i_{2}\cdots i_{m}]$ by
\begin{equation}
[i_{1}i_{2}\cdots i_{m}]=\left\{\begin{array}{lc}
z,   &\mbox{if} ~m=1, \\
x_{i_{1}}x_{i_{m}}x_{i_{m-1}}\cdots x_{i_{1}}, &\ \ \ \ \ \  \mbox{if} ~
1<m\leq n.
\end{array}
\right.
\end{equation}
 It is easy to see that $\theta_n([i_{1}i_{2}\cdots
i_{m}])=(i_{1}i_{2}\cdots i_{m})$, $\theta_{n+1}(x_i)=(i,n+1)$. Therefore
each element of $\widetilde{S}_n$ is of the form
$$z^p[i_1 \cdots i_m][j_1\cdots j_k]\cdots,$$
where $\{i_1 \cdots i_m\}$, $\{j_1 \cdots
j_k\}$, $\cdots$ is a partition of the set $\{1,2,\cdots,n \}$ and
$p\in \mathbb{Z}_2$. 

Let $\lambda=(\lambda_1, \cdots, \lambda_l)$ be a partition of the
positive integer $n$.
We identify $\lambda$ with its Ferrers diagram which
is formed by the array of $n$ dots having $l$ left-justified rows with row $i$
containing
$\la_i$ dots for $1\leq i\leq l.$
A Young tableau $T_{\lambda}$ of shape $\lambda$ is a numbering of the dots of
the Ferrers diagram by $1,2,\cdots, |\lambda|$. For such a
tableau $T_{\lambda}$ of shape $\lambda$ with the numbering
$a_{ij}$ for the $(i, j)$-dot,
we define the element $t_{\lambda}=[a_{11}\cdots
a_{1\lambda_1}][a_{21}\cdots a_{2\lambda_2}]\cdots [a_{l1} \cdots
a_{l\lambda_l}]$ of $\widetilde{S}_n$. We also denote its image in $S_n$ by
\begin{equation*}
\sigma_{\lambda}
=\theta_n(t_{\lambda})=\prod_{i=1}^l(a_{i1}\cdots a_{i\lambda_i})\in S_n.
\end{equation*}

\subsection{The spin group $\widetilde{\Gamma}_n$.}
We will mainly consider the case of the cyclic group
$\Gamma=\langle a|a^{r+1}=1\rangle\simeq \mathbb Z_{r+1}$. But for the most part of this section
we allow $\Gamma$ to be a general finite group. We
denote by $\Gamma_*$ the set of
conjugacy classes of $\Gamma$ and $\Gamma^*=\{\gamma_i|~
i=0,\cdots,r\}$ the set of irreducible characters of $\Gamma$ with
$\gamma_0$ being the trivial character.

Let $\zeta_c$ be the order of the centralizer of an element in the class $c\in
\Gamma_*$, then the order of $c$ is $|\Gamma|/\zeta_c$. When $\Gamma=\mathbb Z_{r+1}$,
$\zeta_c=r+1$ and $|c|=1$ for any $c\in\Gamma_*$.

For $n\in\mathbb Z_+$, let $\Gamma^n$ be the direct product $\Gamma\times\cdots\times
\Gamma$, where $\Gamma^0=1$. The spin group
$\widetilde{S}_n$ acts on $\Gamma^n$ via the permutation action of $S_n$, thus
\begin{equation}
\begin{split}
t_{\la}(g_1,\cdots,g_n)=&(g_{\sigma_{\la}^{-1}(1)},\cdots,g_{\sigma_{\la}^{-1}(n)}),\\
z(g_1,\cdots,g_n)=&(g_1,\cdots,g_n).
\end{split}
\end{equation}
The spin wreath product $\widetilde{\Gamma}_n$ 
is the semi-direct product
$$\widetilde{\Gamma}_n=\Gamma^n\ltimes \widetilde{S}_n=\{(g,t)|g=(g_1,\cdots,g_n)\in\Gamma^n, t\in\widetilde{S}_n\}$$
with the multiplication
$$(g,t)\cdot(h,s)=(gt(h),ts).$$
 Similarly, $\Gamma_n$ is
defined to be  the semi-direct product of $\Gamma^n$ by $S_n$. It
is known that $\widetilde{\Gamma}_n$ is a central extension of
$\Gamma_n$ by $\mathbb{Z}_2$, thus
$|\widetilde{\Gamma}_n|=2n!|\Gamma|^n.$

Let $d$ be the {\it parity} homomorphism from the spin group $\widetilde{S}_n$ to the
group $\mathbb{Z}_2$  by
\begin{equation}
d(t_i)=1~(i=1,\cdots,n-1),\ \ \ \ d(z)=0.
\end{equation}
Similarly, we define a {\em parity} for
$\widetilde{\Gamma}_n$ by
\begin{equation}
d(g,t_i)=1 ~(i=1,\cdots,n-1),\ \ \ d(g,z)=0.
\end{equation}

\subsection{Partition valued functions.}

We recall some basic notions of partitions to describe conjugacy classes
of $\Gamma_n$.
Let $\la=(\la_1,\la_2,\cdots,\la_l)$ be a partition of $n$ with
$\lambda_1 \geq\cdots\geq \lambda_l\geq 1.$ We denote by $l=l(\la)$
the length of the partition $\la$ and set
$|\la|=\la_1+\cdots+\la_l.$ Sometimes we write
$\la=(1^{m_1}2^{m_2}3^{m_3}\cdots)$, where $m_i$ is the multiplicity of $i$ among
the parts of $\la$.

 Given a
finite set $X$, let
$\rho=(\rho(x))_{x\in X}$ be a family of partitions indexed by
$X$, we denote by $l(\rho)=\sum_{x\in X}l(\rho(x))$ the length
of $\rho$ and by $||\rho||=\sum_{x\in
X}|\rho(x)|$ the sum of parts of $\rho$, and then
$\rho=(\rho(x))_{x\in X}$ is called a {\it partition valued function}
on $X.$ Let $ \mathcal
{P}(X)$ be the set of all partitions indexed by $X$ and
$\mathcal{P}_n(X)$ the set of all partitions in $\mathcal {P}(X)$
such that $||\rho||=n.$  For two partition valued functions
$\rho=(\rho(x))_{x\in X}$ and $\sigma=(\sigma(x))_{x\in X}$,
we define the {\it union} of $\rho\cup\sigma$ to be the
partition valued function given by $(\rho\cup\sigma)(x)=\rho(x)\cup\sigma(x)$.
Here the union of two ordinary partitions is taken to be the juxtaposition of
two partitions with their parts rearranged.
Subsequently, $||\rho\cup\sigma||=||\rho||+||\sigma||$
and $l(\rho\cup\sigma)=l(\rho)+l(\sigma)$. A partition valued function is said to be
{\it decomposable} if it is a (non-trivial) union of two or more partition valued functions.

A partition $\la=(\la_1,\la_2,\cdots,\la_l)$ is called {\em
{strict}} if $\la_i\neq \la_j$ for $i\neq j$. We denote by $\mathcal {SP}(X)$ the set of
partition valued functions $(\rho(x))_{x\in X}$ in $ \mathcal
{P}(X)$ where each partition $\rho(x)$ is strict. Let $\mathcal
{OP}(X)$ be the set of partition valued functions $(\rho(x))_{x\in
X}$ in $ \mathcal {P}(X)$ such that all parts of the partitions
$\rho(x)$ are odd integers.

For each partition $\la$ we define the {\em parity} $d(\la)=|\la|-l(\la)$.
Similarly, for a partition valued function $\rho=(\rho(x))_{x\in
X}$, we define $d(\rho)= ||\rho||-l(\rho)$. Then $\rho$ is {\it even} (or {\it odd}) if
$d(\rho)$ is even (or odd). We let ${\mathcal{P}}_n^0(X)$ (or ${\mathcal{P}}_n^1(X)$) to be the collection
of even (or odd) partition valued functions on $X$.

As convention we set $\mathcal {SP}_n^i(X)=\mathcal
{P}_n^i(X)\cap \mathcal {SP}(X)$ and $\mathcal {OP}_n(X)=\mathcal
{P}_n(X)\cap \mathcal {OP}(X)$ for $i\in \{0,1\}$. For simplicity,
$\mathcal{P}(X)$ will be simply written as $\mathcal {P}$ when $X$
consists of a single element. Similarly we have notations
such as $\mathcal {OP}$, $\mathcal {SP}$, $\mathcal {OP}_n$ and $\mathcal
{SP}_n^i$.

\subsection{Split conjugacy classes of $\Gamma^n\ltimes S_{\mu}$.}
We first recall the parametrization of conjugacy classes of $\Gamma_n$ by partition valued functions. For
an element $(g, \sigma)\in \Gamma_n$, write the permutation $\sigma$ as
a product of disjoint cycles. For each cycle
$(i_1i_2\cdots i_k)$ inside $\sigma$, we associate the cycle-product $g_{i_k}g_{i_{k-1}}\cdots
g_{i_1}\in \Gamma$. Now for each conjugacy class $c$ let $m_k(c)$ be the multiplicity of $k$
such that the cycle product $\in c$. The resulted partition valued function $\rho\in\mathcal P(\Gamma_*)$, where $\rho(c)=(1^{m_1(c)}2^{m_1(c)}\cdots)$,
determines the conjugacy class of $(g, \sigma)$ completely \cite{M}.

For a partition $\mu$ of $n$ 
we define the Young subgroup of $S_n$ to be
$$S_{\{1,\cdots,\mu_1\}}\times \cdots\times S_{\{\mu_1+\cdots+\mu_{s-1}+1, \cdots, \mu_1+\cdots+\mu_s\}},$$
which will be abbreviated as $S_{\mu_1}\times\cdots\times S_{\mu_s}$. Similarly $\Gamma^n \ltimes (S_{\mu_1} \times S_{\mu_2}\times
\cdots \times S_{\mu_s})\simeq \Gamma_{\mu_1}\times\cdots \times\Gamma_{\mu_s}:=\Gamma_{\mu}$ is a subgroup of $ \Gamma_n$, also
called the Young subgroup of $\Gamma_n$ associated to $\mu$.

Now we discuss the parametrization of conjugacy classes
generated by Young subgroups.  Let $x=(g,\omega)$ be an
element of the Young subgroup
$\Gamma_{\mu_1}\times \Gamma_{\mu_2}\times\cdots \times
\Gamma_{\mu_s}$, 
where $x=x_1\cdots x_s$ and $x_i=(g^{(i)},\sigma_i)\in \Gamma_{\mu_i}$.
Let $\rho^i$ be the partition valued function on $\Gamma_*$ given by the congugacy class of $x_i$, thus $||\rho^i||=\mu_i$. We remark that if $x_i$ is viewed as an element of $\Gamma_n$, then the conjugacy class of $x_i$ corresponds to the
partition valued function $\rho^i\cup (1^{n-\mu_i})$. Then
$\rho=\rho^1\cup \rho^2\cup \cdots \cup \rho^s$
will be the partition valued function of $(g, \omega)$.
In this way we define a bijection $\phi$ from the decomposable
partition valued functions $\rho=\rho^1\cup\cdots\cup\rho^s$ such that $||\rho^i||=\mu_i$ to the conjugacy classes of
$x=(g,\omega)$ in $\Gamma_{\mu_1}\times
\Gamma_{\mu_2}\times\cdots \times \Gamma_{\mu_s}$.

For
$\rho=(\rho(c))_{c\in\Gamma_*}\in \mathcal {P}_n(\Gamma_*)$, let $C_{\rho}$ be the corresponding conjugacy class in $\Gamma_n$. Let $c^0,\cdots,c^r$ be the conjugacy classes
of $\Gamma$, here $c^0=\{1\}$, the trivial class. Let
$T_{\rho(c^i)}$ be the standard Young tableau such that the numbers $
\sum_{j=1}^{i-1}|\rho(c^j)|+1,\cdots,\sum_{j=1}^{i}|\rho(c^j)|$ are
placed in the Young diagram of shape
$\rho(c^i)=(\rho(c^i)_1,\cdots,\rho(c^i)_l)$ from the left to the right and
from the first row to
the last row. Then we get
\begin{equation}\label{eq:preimage}
\begin{split}t_{\rho(c^i)}=&[a_{i-1}+1,\cdots,
a_{i-1}+\rho(c^i)_1]\cdots \\
&[a_{i-1}+\rho(c^i)_1+\cdots+\rho(c^i)_{l-1},\cdots, a_{i-1}+|\rho(c^i)|],
\end{split}
\end{equation}
where $a_{i-1}=\sum_{j=0}^{i-1}|\rho(c^j)|$. Finally, we  define $t_{\rho}=t_{\rho(c^0)}t_{\rho(c^1)}\cdots
t_{\rho(c^r)}$ in $\widetilde{S}_n$. For any permutation $\sigma\in S_n$,
we also define $t_{\rho}^{\sigma}$ to be the
element obtained from $t_{\rho}$ by
permuting the natural numbering by $\sigma$. Thus the general element of $\widetilde{\Gamma}_n$ is of
the form $(g,z^pt^{\sigma}_{\rho})$, where $\rho$ is the type of the
conjugacy class of $(g,z^pt^{\sigma}_{\rho})$ and $\sigma\in S_n$.

 An element $\widetilde{x}\in\widetilde{\Gamma}_n$ is called
{\em non-split} if $\widetilde{x}$ is conjugate to $z\widetilde{x}$.
Otherwise $\widetilde{x}$ is said to be {\em split}. A conjugacy
class of $\widetilde{\Gamma}_n$ is called split if its elements are
split. Correspondingly an element $x\in\Gamma_n$ is called split if
$\theta^{-1}_n(x)$ is split. Therefore a conjugacy class $C_{\rho}$
of $\Gamma_n$ splits if and only if the preimage
$\theta_n^{-1}(C_{\rho})\triangleq D_{\rho}$ splits into two
conjugacy classes in $\widetilde{\Gamma}_n$. Read has proved that
the preimage
$\theta_n^{-1}(C_{\rho})$ splits into two conjugacy classes in
$\widetilde{\Gamma}_n$ if and only if $\rho\in \mathcal
{OP}_n(\Gamma_*)$ or $\rho\in \mathcal {SP}_n^1(\Gamma_*)$
(cf. \cite{FJW}).
 For each split
conjugacy class $C_{\rho}$ in $\Gamma_n$, we define the conjugacy
class $D_{\rho}^{+}$ in $\widetilde{\Gamma}_n$ to be the conjugacy
class containing the element $(g,t_{\rho})$ and define
$D_{\rho}^{-}=zD_{\rho}^{+}$, then $ D_{\rho}=D_{\rho}^{+}\cup
D_{\rho}^{-}$.

For a partition $\la=(1^{m_1}2^{m_2}3^{m_3}\cdots)$ of $n$, we denote by
$z_{\la}=\prod_{i\geq 1}i^{m_i}m_i!$ the order of the centralizer of
an element with cycle type $\la$ in $S_{n}.$
For each partition valued function
$\rho=(\rho(c))_{c\in\Gamma_*}$, we find that
\begin{equation}
Z_{\rho}=\prod_{c\in\Gamma_*}z_{\rho(c)}\zeta_c^{l(\rho(c))}
\end{equation} is the order of the centralizer of an element of
conjugacy type $\rho=(\rho(c))_{c\in\Gamma_*}$ in $\Gamma_n$. Correspondingly
the order of the centralizer of an element of conjugacy type $\rho$ in $\widetilde{\Gamma}_n$ is given by
\begin{equation}\label{eq:1}
\widetilde{Z}_{\rho}=\left\{ \begin{aligned}
         2Z_{\rho}, & \ \ \ C_{\rho} ~\hbox{is split,}  \\
                  Z_{\rho},& \ \ \ C_{\rho} ~\hbox{is non-split.}
                          \end{aligned} \right.
                          \end{equation}
Following the usual definition \cite{S}
a representation $\pi$ of $\widetilde{\Gamma}_n$ is called {\em
spin} if $\pi(z)=-1$. In particular, the character values of a
spin representation are determined by its values on the split
classes, since in that case,
$Tr(\pi(z\widetilde{x}))=-Tr(\pi(\widetilde{x}))=0$ whenever
$\widetilde{x}$ and $z\widetilde{x}$ are conjugate in
$\widetilde{\Gamma}_n$.

Let $(-1)^d$ be the sign representation of $\widetilde{\Gamma}_n$:
$\widetilde{x}\longrightarrow (-1)^{d(\widetilde{x})}.$  When $(-1)^d\pi\simeq\pi$ we call
$\pi$  a {\em double spin} representation of $\widetilde{\Gamma}_n$. If
$\pi^{'}=(-1)^d\pi\ncong \pi,$ then $\pi^{'}$ and $\pi$ are called
a pair of {\em associate spin} representations of $\widetilde{\Gamma}_n$.

\section{Twisted Grothendieck groups}

In this section we recall some fundamental facts about supermodules, spin super functions and the irreducible spin
 characters of $\widetilde{S}_n$, then we study the spin
representations of $\Gamma^n \ltimes \widetilde{S}_{\mu}$, which is a double cover
of $\Gamma_{\mu}$.
\subsection{Supermodules}

Let $\mathbb{C}[\widetilde{\Gamma}_n]$ be the group algebra of $\widetilde{\Gamma}_n,$ then $\mathscr{A}_n=\mathbb{C}[\widetilde{\Gamma}_n]/(1+z)$ becomes a $\mathbb{Z}_2$-graded algebra by setting
$deg(t_{i})=1,(i=1,\cdots,n-1)$. A $\widetilde{\Gamma}_n$-module $V$ is called a {\it spin module}
if $z$ acts as $-id_V$, then $V$ can also be viewed as
an $\mathscr{A}_n$-module. Conversely any $\mathscr{A}_n$-module
is also a spin module of $\widetilde{\Gamma}_n$. If $V=V_0\oplus V_1$ and $\mathscr{A}_n^iV_j\subset V_{i+j}$,
where $\mathscr{A}_n^i$ is the $i$th homogeneous subspace, then $V$ is called a supermodule.
It is well-known that complex simple superalgebras have only two types \cite{Jo}:

(1) Type $M$. The superalgebra $M(r|s)$ is equal to the matrix algebra $Mat(r+s, r+s)$, where
the matrices are partitioned as $2\times 2$ block matrices so that the main diagonals are $r\times r$ and $s\times s$ sub-matrices. The subspace $M(r|s)_0$ consists of diagonal block matrices, and the
 subspace $M(r|s)_1$ is equal to the space of off-diagonal block matrices.

(2) Type $Q$. The superalgebra $Q(n)$ is the subalgebra of $M(n|n)$ formed by block matrices
with both equal diagonal block matrices and skew diagonal block matrices.

Furthermore, $\mathscr{A}_n$ is semisimple, so it is a direct product of finitely many simple superalgebras.
A supermodule is said to be of type $M$ (or $Q$) if it is a multiple of one minimal left superideal of $M(r|s)$
(or $Q(n)$).
Subsequently any finite dimensional $\mathbb{C}[\Gamma_n]$-supermodule is
isomorphic to a direct sum of simple supermodules of type $M$ or type $Q$.

If $V$ is a double spin irreducible $\mathbb C[\widetilde{\Gamma}_n]$-module,
then $V$ is already an $\mathscr{A}_n$-supermodule by the inherited action.
If $V$ is an irreducible associate spin $\mathbb C[\widetilde{\Gamma}_n]$-module, then
$D(V)=V\oplus V'$ becomes an irreducible $\mathscr{A}_n$-supermodule where
$D(V)_0=\{(v, v)|v\in V\}$, $D(V)_1=\{(v, -v)|v\in V\}$ and
the action is induced from that of the ordinary module, i.e.
$g^{(i)}(u, v)=(g^{(i)}u, (-1)^ig^{(i)}v)$ for $g^{(i)}\in\mathscr{A}_n^{(i)}$, the degree
$i$-subspace of $\mathscr{A}_n^{(i)}$.

In the following we will use supermodules to compute irreducible characters.
The underlying principle is that an irreducible (spin) supermodule of $\widetilde{\Gamma}_n$
remains irreducible as an (spin) module when it is of type $M$ or decomposes into two
irreducible (spin) modules when it is of type $Q$, and any
irreducible spin module can be realized in this way.

 \subsection{The space $ {R}^-(\widetilde{\Gamma}_n)$}

A spin class function on $\widetilde{\Gamma}_n$ is a class function
from $\widetilde{\Gamma}_n$ to $\mathbb{C}$ such that
$f(zx)=-f(x)$, thus spin class functions vanish on non-split
conjugacy classes. A spin {\it super} class function on
$\widetilde{\Gamma}_n$ is a spin class function $f$ on
$\widetilde{\Gamma}_n$ such that $f$ vanishes further on odd strict
conjugacy classes. Let $ {R}^{-}(\widetilde{\Gamma}_n)$ be the
$\mathbb{C}$-span of spin
super class functions on
$\widetilde{\Gamma}_n$.

The twisted product $\widetilde{\Gamma}_l\tilde{\times}\widetilde{\Gamma}_m$ is
equal to $\widetilde{\Gamma}_l\times\widetilde{\Gamma}_m$ as a set but with
the multiplication
$$(t,t^{'})(s,s^{'})=(tsz^{d(t^{'})d(s)},t^{'}s^{'}),$$
where $s,t\in \widetilde{\Gamma}_l,$ $s^{'},t^{'}\in
\widetilde{\Gamma}_m$ are homogeneous. We define the spin direct
product \cite{FJW} of $\widetilde{\Gamma}_l$ and $\widetilde{\Gamma}_m$ by
\begin{equation}\label{s_n}
\widetilde{\Gamma}_l \hat{\times}\widetilde{\Gamma}_m=\widetilde{\Gamma}_l\tilde{\times}\widetilde{\Gamma}_m/\{(1,1),(z,z)\},
\end{equation}
which can be embedded into the spin group $\widetilde{\Gamma}_{l+m}$ canonically by letting
\begin{equation}
(t_i^{'},1)\longmapsto t_i,\ \ \ \ (1,t_j^{''})\longmapsto t_{l+j},
\end{equation}
where $t_i^{'}\in \widetilde{\Gamma}_l \,(i=1,\cdots,l-1), t_j^{''}\in \widetilde{\Gamma}_m
\,(j=1,\cdots,m-1)$. We identify $\widetilde{\Gamma}_l\hat{\times}\widetilde{\Gamma}_m$
with its image in $\widetilde{\Gamma}_{l+m}$ and regard it as a subgroup of $\widetilde{\Gamma}_{l+m}$.

We remarked earlier that we would study spin modules via supermodules \cite{Jo}. The following exposition
of twisted Grothendieck rings of supermodules follows \cite{FJW} closely. For two spin supermodules $U$ and $V$ of $\widetilde{\Gamma}_{l}$ and $\widetilde{\Gamma}_m$, we define the
super (outer)-tensor product $U{\hat{\otimes}} V$
by
$$(t,s)(u{\hat{\otimes}
} v)=(-1)^{d(s)d(u)}(tu{\hat{\otimes}} sv),$$
where $s$ and $u$ are homogeneous elements. Then $U\hat{\otimes}V$ is a spin
$\widetilde{\Gamma}_l\hat{\times}\widetilde{\Gamma}_m$-supermodule.
Moreover, let $U$ and $V$ be irreducible supermodules for $\widetilde{\Gamma}_{l}$ and $\widetilde{\Gamma}_m$ respectively, then

 (1) if both $U$ and $V$ are of type $M$, then $U\hat{\otimes} V$ is a simple $\widetilde{\Gamma}_l\hat{\times}\widetilde{\Gamma}_m$
 -supermodule of type $M;$

 (2) if  $U$ and $V$ are of different types, then $U\hat{\otimes }V$ is a simple
  $\widetilde{\Gamma}_l\hat{\times}\widetilde{\Gamma}_m$-supermodule of type $Q;$

 (3) if both $U$ and $V$ are of type $Q$,
  then $U\hat{\otimes} V\simeq N\oplus N$ for some simple $\widetilde{\Gamma}_l\hat{\times}\widetilde{\Gamma}_m$-supermodules $N$ of type $M.$

The irreducible summands in case (3) are more subtle as ordinary irreducible modules. In fact we have the following
result. First of all, if $V$ is an irreducible double spin module, then $V$ is {\it a priori}
an irreducible supermodule
of type M (we still use the same symbol for the supermodule). If $V$ is an irreducible associate spin module, then $D(V)=V\oplus V'$ is an irreducible supermodule
of type Q.  The following result is mostly from \cite{Jo}.

\begin{prop}\label{tensorprod}
Let $f_1$ and $f_2$ be the spin characters afforded by an irreducible
$\widetilde{\Gamma}_m$-module $V_1$ and an $\widetilde{\Gamma}_n$-module $V_2$ respectively.

(i) If both $V_i$ are double spin, then the tensor product $V_1\hat{\otimes}V_2$ is
irreducible both as a supermodule and as an ordinary module
for $\widetilde{\Gamma}_m\hat{\times}\widetilde{\Gamma}_n$.

(ii) If $V_1$ is double spin and $V_2$ is associate spin,
then the tensor product $V_1\hat{\otimes}D(V_2)$ is irreducible as a $\widetilde{\Gamma}_m\hat{\times}\widetilde{\Gamma}_n$-supermodule and decomposes
into $V_1\circledast V_2\oplus (V_1\circledast V_2)'$ as an ordinary module, where $(V_1\circledast V_2)'$
is the associated module of the irreducible module $V_1\circledast V_2$.

(iii) If both $V_i$ are associate spin,
then the tensor product $D(V_1)\hat{\otimes}D(V_2)$ decomposes
into $W\oplus W$, where $W$ is an irreducible $\widetilde{\Gamma}_m\hat{\times}\widetilde{\Gamma}_n$-supermodule
of type M. Set $W=V_1\circledast V_2$ when it is viewed as an ordinary irreducible module
(up to isomorphism), then the character $f_1\circledast f_2$ of the irreducible summand $V_1\circledast V_2$ satisfies that
\begin{align*}
f_1\circledast f_2(x_1, x_2)=\begin{cases} 2(\sqrt{-1})^{d(x_1)d(x_2)}f_1(x_1)f_2(x_2) & \mbox{both $f_i$ are associate spin,}\\
f_1(x_1)f_2(x_2) & \mbox{otherwise,}
\end{cases}
\end{align*}
\end{prop}
where $x_1\in\widetilde{\Gamma}_m$, $x_2\in\widetilde{\Gamma}_n$.

\begin{proof} The statements on tensor products of supermodules are clear.
The relationship between supermodules and modules are proved in \cite{Jo}.
The last relation about characters follows from
analysis of basic spin characters, see \cite{Jo} or
\cite{J} for details. Another treatment can be found in
\cite{St} for the special case of $\widetilde{S}_n$. 
\end{proof}

The starred tensor product $V_1\circledast V_2$ is essentially
Schur's tensor product of spin modules \cite{S}. We can generalize the starred tensor product
$V_1\circledast\cdots\circledast V_s$ for multiple
modules. Let $f_i$ ($i=1, \cdots, k$) be the character of the irreducible
associate spin module $V_i$, which comes from an irreducible supermodule of type Q. Let
$f_j$ ($j=k+1, \cdots, s$) be the character of the irreducible
associate spin module $V_j$, which comes from an irreducible supermodule of type M. We define $f_1\circledast\cdots \circledast f_s$
to be the character of the irreducible component $V_1\circledast\cdots \circledast V_s$ (as an ordinary
module) in the super tensor
product
\[
D(V_1)\hat{\otimes}\cdots \hat{\otimes}D(V_k)\hat{\otimes} V_{k+1}\hat{\otimes}\cdots \hat{\otimes}V_s.
\]
Note that $V_1\circledast\cdots \circledast V_s$ is only defined up to isomorphism. See \cite{K}
for a similar discussion for supermodules.
Using induction on $k$ we also have

\begin{equation}\label{eqf}
f_1\circledast\cdots\circledast f_s({x}_1\cdots {x}_s)
=2^{[\frac{k}2]} (\sqrt{-1})^{[\frac{k}2]d(x_1)\cdots d(x_k)}f_1({x}_1)\cdots f_s({x}_s)
\end{equation}
where $x_i\in\widetilde{\Gamma}_{\mu_i}$ and $[a]$ denotes the maximum integer
$\leq a$.

 The twisted Grothendieck group $R^{-}({\Gamma})=\bigoplus_{n\geq 0} R^{-}(\widetilde{\Gamma}_n)$
 has an associative algebra structure. The multiplication 
 is defined as follows.
Let $f\in R^{-}(\widetilde{\Gamma}_l)$, $g\in R^{-}(\widetilde{\Gamma}_m) $.
Then $f\times g$ is an element of $R^{-}(\widetilde{\Gamma}_l\hat{\times} \widetilde{\Gamma}_m)$
and we define
$$f\circ g=\mbox{Ind}_{\widetilde{\Gamma}_l\hat{\times}\widetilde{\Gamma}_m}^{\widetilde{\Gamma}_{l+m}}(f\times g)$$
which is an element of $R^{-}(\widetilde{\Gamma}_{l+m})$. This gives a bilinear multiplication
on $R^{-}({\Gamma})$. It follows from \cite{FJW} that $R^{-}({\Gamma})$ becomes a graded associative $\mathbb{C}$-algebra.

The irreducible spin super characters of
$\widetilde{\Gamma}_n$ form a $\mathbb{C}$-basis of
$R^{-}(\widetilde{\Gamma}_n)$. For two simple supermodules $\phi, \varphi\in
R^{-}(\widetilde{\Gamma}_n)$, the standard
inner product can also be used for supermodules and we have that \cite{FJW}
\begin{equation}
\langle\phi,\varphi\rangle=\left\{\begin{array}{lll}
1  ~& \hbox{if $\phi\simeq\varphi$ is type $M$},\\
2 ~ & \hbox{if $\phi\simeq\varphi$ is type $Q$}~,~~~~\\
0 ~& \hbox{otherwise}.
\end{array}
\right.
\end{equation}

For a simple supermodule $V$ we define
\begin{equation}
\dot{c}=c(V)=\left\{\begin{array}{lc}0  ~&
 \hbox{if $V$ is type  $M$},\\
1~ &\hbox{if $V$ is type $Q$},
\end{array}
\right.
\end{equation}
and then we extend the definition to multiple copies of $V$
by $c(V^{\oplus n})=c(V)$.

 Let $f_{i}$  $(i=1,\cdots, s)$ be the spin characters of irreducible $\widetilde{\Gamma}_{\mu_i}$-modules $V_i$.
 Assume that $k$ of them are associate spin modules, say, $V_1, \cdots, V_k$ are associate spin and
 $V_{k+1}, \cdots, V_s$ are double spin modules. Then $D(V_1), \cdots, D(V_k)$ are irreducible supermodules
 of type Q, and $V_{k+1}, \cdots, V_s$ are irreducible supermodules of type M.

Let $f_1\circ\cdots\circ f_s$ be the induced character of $f=f_1\circledast\cdots\circledast f_s$ from $\widetilde{\Gamma}_{\mu}$ to $\widetilde{\Gamma}_n$. For later applications we determine the restriction $\mbox{Res}_{\widetilde{\Gamma}_{\mu}}(f_1\circ\cdots\circ f_s)$.
  We denote by $\widetilde{\Gamma}_{\mu}\backslash \widetilde{\Gamma}_{n}/\widetilde{\Gamma}_{\mu}$
 the collection of the double cosets $\widetilde{\Gamma}_{\mu}t\widetilde{\Gamma}_{\mu}$,
 and set $(\widetilde{\Gamma}_{\mu})_t=t\widetilde{\Gamma}_{\mu}t^{-1}\cap\widetilde{\Gamma}_{\mu}$ for any double coset
 representative $t$. By Mackey's decomposition theorem, we have
$$\mbox{Res}_{\widetilde{\Gamma}_{\mu}}(f_1\circ\cdots\circ f_s)=\bigoplus_{t\in \widetilde{\Gamma}_{\mu}\backslash \widetilde{\Gamma}_{n}/\widetilde{\Gamma}_{\mu}}
\mbox{Ind}_{(\widetilde{\Gamma}_{\mu})_{t}}^{\widetilde{\Gamma}_{n}}(f^t),$$
where $f^t(\tilde{x})=f(t^{-1}\tilde{x}t)$.
From Frobenius reciprocity it follows that
\begin{equation}\nonumber
\begin{split}\langle f_1\circ\cdots\circ f_s,f_1\circ\cdots\circ f_s\rangle=&\langle f,\mbox{Res}_{\widetilde{\Gamma}_{\mu}}(f_1\circ\cdots\circ f_s)\rangle\\
=&\sum_{t\in \widetilde{\Gamma}_{\mu}\backslash\widetilde{\Gamma}_{n}/\widetilde{\Gamma}_{\mu}}\langle \mbox{Res}_{(\widetilde{\Gamma}_{\mu})_t}(f),f^t\rangle_{(\widetilde{\Gamma}_{\mu})_t}.
\end{split}
\end{equation}

By definition $f=f_1\circledast\cdots\circledast f_s$ is an irreducible spin
 character. If $f_1\circ\cdots\circ f_s$ is a spin $\widetilde{\Gamma}_n$-irreducible character, then
we have $\langle \mbox{Res}_{(\widetilde{\Gamma}_{\mu})_t}(f),f^t\rangle_{(\widetilde{\Gamma}_{\mu})_t}=0$ for $t\neq 1$
(nontrivial double coset). In fact when $t=1$,
 one has $\langle \mbox{Res}_{(\widetilde{\Gamma}_{\mu})_t}(f),f^t\rangle_{(\widetilde{\Gamma}_{\mu})_t}
 =\langle f, f\rangle_{\widetilde{\Gamma}_{\mu}}=1$,
 and $\langle f_1\circ\cdots\circ f_s,  f_1\circ\cdots\circ f_s\rangle=1$.
 Therefore in this case we have
\begin{equation}\label{Inner}
\begin{split}
 & \langle f_1\circ\cdots\circ f_s,  f_1\circ\cdots\circ f_s\rangle_{\widetilde{\Gamma}_{n}}
 =\langle f, f\rangle_{\widetilde{\Gamma}_{\mu}}\\
=&2^{k-\dot{c}}\frac{1}{|\widetilde{\Gamma}_{\mu}|}\sum_{\tilde{x}\in \widetilde{\Gamma}_{\mu}}
 f_1(\tilde{x}_1)\cdots  f_s(\tilde{x}_s)\overline{f_1(\tilde{x}_1)\cdots  f_s(\tilde{x}_s)},
\end{split}
\end{equation}
where $\tilde{x}=\tilde{x}_1\cdots \tilde{x}_s$ and $\tilde{x}_i\in
\widetilde{\Gamma}_{\mu_i} (i=1,\cdots,s)$.
\subsection{Irreducible spin representations of $\widetilde{S}_n$}

To compute irreducible characters of
$\widetilde{S}_n$, Schur \cite{S} introduced the symmetric functions
$Q_{\nu}\in\mathbb Q[p_1, p_3, \ldots]$, where $\nu\in\mathcal{SP}_n$ and $p_k=
\sum_{i\geq 1}x_i^k$ is the $k$th power sum symmetric function.
For $l=l(\nu)\leq n$, the Schur $Q$-function $Q_{\nu}$ is given by

$$Q_{\nu}(x_1, \cdots, x_n)
=2^{l}\sum_{\alpha\in S_n/S_{n-l}}x_{\alpha(1)}^{\nu_1}\cdots
x_{\alpha(n)}^{\nu_n} \prod_{\nu_i>\nu_j}\frac{x_{\alpha(i)}+x_{\alpha(j)}}{x_{\alpha(i)}-x_{\alpha(j)}},
$$
where $\nu_m=0$ for $m>l(\nu)$.
Schur showed that
for each $\nu\in \mathcal {SP}_n$ there corresponds a unique irreducible (double) spin
character $\Delta_{\nu}$ if $n-l(\nu)$ is even or a pair of irreducible (associate) spin characters $\Delta_{\nu}^{+}$
and $\Delta_{\nu}^{-}$ if $n-l(\nu)$ is odd.
The spin character values  $\{\Delta^{\la}_{\nu}|\la\in\mathcal
{OP}_n\}$ 
are determined by
\begin{equation}\label{Q}
Q_{\nu}=\sum_{\la\in\mathcal {OP}_n}2^{\frac{l({\nu})+l(\la)+\bar{d}(\nu)}{2}}
z_{\la}^{-1}\Delta_{\nu}^{\la}p_{\la},
\end{equation}
where $p_{\la}=p_{\la_1}p_{\la_2}\cdots p_{\lambda_l}$
are the power sum symmetric functions, $(\Delta_{\nu}^+)^{\la}=\Delta_{\nu}^{\la}$ for odd $n-l(\nu)$ and

\begin{equation*}
\bar{d}(\nu)=\left\{ \begin{aligned}
         0\ \ \  &  \hbox{~when~} d(\nu) \hbox{~is ~even}, \\
                 1 \ \ \ & \hbox{~when~} d(\nu) \hbox{~is ~odd}.
                          \end{aligned} \right.
                          \end{equation*}
 \begin{theorem}(Schur \cite{S}) For each ${\nu}=({\nu}_1,\cdots, {\nu}_l)\in \mathcal
 {SP}_n (n\geq 4)$, the corresponding irreducible spin characters of $\widetilde{S}_n$ are determined as follows.

 (i) If $n-l$ is even, 
 there is a unique (double) spin irreducible character
$\Delta_{\nu}$ whose character values $\Delta^{\la}_{\nu}\, (\la\in\mathcal {OP}_n)$
are given by (\ref{Q})
and $\Delta^{\mu}_{\nu}=0$ for $\mu\notin\mathcal {OP}_n.$

(ii) If $n-l$ is odd, 
there are two irreducible (associate) spin characters
$\Delta_{\nu}^{+},~\Delta_{\nu}^{-}$. The character values $(\Delta_{\nu}^{+})^{\la}$
are given by (\ref{Q}) for $\la\in\mathcal {OP}_n$,
and for other classes they are given by
$$
(\Delta_{\nu}^{+})^{\nu}=(\sqrt{-1})^{(n-l({\nu})+1)/2}\sqrt{{\nu}_1\cdots {\nu}_l/2},$$
and
$(\Delta_{\nu}^{+})^{\mu}=0$ for $\mu\neq {\nu}\in \mathcal {SP}^1_n$.
Moreover, $(\Delta_{\nu}^{-})^{\mu}=(\Delta_{\nu}^{+})^{\mu}$ for $\mu$ even
and $(\Delta_{\nu}^{-})^{\mu}=-(\Delta_{\nu}^{+})^{\mu}$ for $\mu$ odd.

\end{theorem}

For an iterative method to compute the spin characters of $\widetilde{S}_n$ and
an explicit character table up to degree 13, see \cite{Mo}.

Let $V_i$ be the
$i$th irreducible $\Gamma$-module affording the character $\gamma_i$, $i\in\{0, 1, \ldots, r\}$.  Let $\Omega=\mathcal P_n(l\leq r+1)$, the set of partitions of $n$ with lengths $\leq r+1$.
Now take distinct integers $i_1,\ldots,i_s$ from $\{0,1, \ldots,r\}$, a partition
$\mu=({\mu_1}, \ldots, {\mu_s})\in\Omega$, and let $W$ be a spin
supermodule of $\widetilde{S}_{\mu}$. Then
the tensor product $V_{i_1}^{\otimes
\mu_1}\otimes\cdots \otimes
V_{i_s}^{\otimes \mu_s}\otimes W$  becomes a spin $ \Gamma^n \ltimes
\widetilde{S}_{\mu}$-supermodule under the action
\begin{equation}
\begin{split}
&(g,z^pt_{\rho})\cdot(v_1\otimes\cdots\otimes v_n \otimes w)\\
=&(g_1v_{\sigma_{\rho}^{-1}(1)}\otimes\cdots\otimes
g_nv_{\sigma_{\rho}^{-1}(n)})\otimes (z^pt_{\rho}w),
\end{split}
\end{equation}
where $(g, z^pt_{\rho})\in\Gamma^n\ltimes \widetilde{S}_n$, $v_1\otimes\cdots\otimes v_n \in V_{i_1}^{\otimes
\mu_1}\otimes\cdots \otimes V_{i_s}^{\otimes \mu_s}$, and $w\in W$.

In particular, when $s=1$ the module $V_i^{\otimes n}\otimes W$ is a spin supermodule
of $\widetilde{\Gamma}_n$.

\section{Irreducible spin character tables of $\widetilde{\Gamma}_n$.}
In this section we will construct the irreducible spin modules of $\widetilde{\Gamma}_n$.
By the general theory of spin characters \cite{FJW} it
is enough to focus on strict partition valued
functions of $\widetilde{\Gamma}_{n}$. We will show that essentially only one special class of conjugacy classes (corresponding to odd strict partition valued
functions) can support nonzero irreducible character values.

\subsection{The irreducible spin supermodules of $\widetilde{\Gamma}_n$}
For $j_1, \cdots,
j_n\in \{0,\cdots,r\}$ and $\Gamma=\mathbb Z_{r+1}$,
let $\gamma=\gamma_{j_1}\otimes\gamma_{j_2}\otimes\cdots\otimes\gamma_{j_n}$.
Then $\gamma$ is an irreducible character of $\Gamma^n$ through the usual tensor product
action.
 Since
$\Gamma^n$ is abelian, they form a group  $X=Hom(\Gamma^n, \mathbb{C}^{*})$ under multiplication. The group
$\widetilde{\Gamma}_n$ acts on $X$ by
$$(\widetilde{x}\cdot\gamma)(g)=\gamma(\widetilde{x}\cdot g) ~~~~\ \ \ \hbox{for}~ \widetilde{x}\in \widetilde{\Gamma}_n, ~\gamma\in X, ~g\in \Gamma^n.$$

In particular the subgroup $\widetilde{S}_n$ acts on $X$. We introduce the {\em class orbits} for the action
of $\widetilde{S}_n$ on $X$. For a partition $\mu\in \Omega=\mathcal{P}_n(l\leq r+1)$ and distinct integers $i_1,\cdots,i_s$ from $\{0,1,\cdots,r\}$ (thus $s$ must be smaller than $r+1$), we denote by $I$ the set of the sequences
$\lfloor i_1, i_2, \cdots, i_s \rfloor$ such that $s=l(\mu)\leq r+1$.  For such $\mu=(\mu_1,\cdots,\mu_s)\in\Omega$ and a sequence
$\lfloor i_1, i_2, \cdots, i_s \rfloor\in I$,
 we associate a $\widetilde{S}_n$-orbit of $X$ as follows:
\begin{equation}
\begin{split}
&\mathcal {O}(\gamma_{i_1}^{\otimes \mu_1}\otimes
\gamma_{i_2}^{\otimes \mu_2}\otimes \cdots \otimes
\gamma_{i_s}^{\otimes\mu_s})\\
=&\{\gamma_{j_1}\otimes\gamma_{j_2}\otimes\cdots\otimes\gamma_{j_n}|
\hbox{~there ~are~} \mu_k \hbox{~indices~equal~to~} i_k \}.
\end{split}
\end{equation}
With fixed $\mu$, the set of these orbits is called a {\em class orbit} with
type $\mu=(\mu_1,  \cdots , \mu_s)$. For simplicity, we denote
 by $\Phi_{\mu}$ the {\em class orbit } as follows:
\begin{equation}
\Phi_{\mu}\triangleq \{\mathcal {O}(\gamma_{i_1}^{\otimes
\mu_1}\otimes \gamma_{i_2}^{\otimes \mu_2}\otimes \cdots \otimes
\gamma_{i_s}^{\otimes \mu_s})~|~ \lfloor i_1, i_2, \cdots,
i_s\rfloor \in I\}.
\end{equation}
Moreover, we say that an irreducible character
$\gamma_{j_1}\otimes\gamma_{j_2}\otimes\cdots\otimes\gamma_{j_n}$
has type $\mu=(\mu_1, \mu_2, \cdots , \mu_s)$ if it is contained in
an orbit $\mathcal {O}(\gamma_{i_1}^{\otimes \mu_1}\otimes
\gamma_{i_2}^{\otimes \mu_2}\otimes \cdots \otimes
\gamma_{i_s}^{\otimes \mu_s})$.

\begin{lem}

(1) For $\mu\in\Omega$, the number of the class orbits $\Phi_{\mu}$ is equal to
$|\Omega|.$

(2) For a partition $\mu=(\mu_1,\cdots,\mu_s)\in
\Omega,$  each {\em class orbit} $\Phi_{\mu}$ contains $K_{\mu}$
orbits, where
 \begin{equation}K_{\mu}=
\begin{bmatrix}1\\r+1\end{bmatrix}\begin{bmatrix}1\\r\end{bmatrix}
\cdots
\begin{bmatrix}1\\r+1-s+1\end{bmatrix}=\frac{(r+1)!}{(r+1-s)!}.
 \end{equation}
 \end{lem}

 For a sequence $\lfloor i_1, i_2, \cdots, i_s\rfloor \in I$,
$\gamma_{i_1}^{\otimes \mu_1}\otimes \gamma_{i_2}^{\otimes
\mu_2}\otimes \cdots \otimes \gamma_{i_s}^{\otimes\mu_s}$ is a
representative of the $\widetilde{S}_n$-orbit $\mathcal {O}(\gamma_{i_1}^{\otimes
\mu_1}\otimes \gamma_{i_2}^{\otimes \mu_2}\otimes \cdots \otimes
\gamma_{i_s}^{\otimes\mu_s})$ in $X.$  For simplicity, we set 
$\gamma_i^{\mu}\triangleq \gamma_{i_1}^{\otimes \mu_1}\otimes
\gamma_{i_2}^{\otimes \mu_2}\otimes \cdots \otimes
\gamma_{i_s}^{\otimes\mu_s}$. For a
partition $\mu \in \Omega$, let $T_{\mu}=\{z^pt_{\rho}\in \widetilde{S}_n | z^p t_{\rho} \cdot
\gamma_{i}^{\mu}= \gamma_{i}^{\mu}\}$, then
\begin{equation}
\begin{split}
T_{\mu}\simeq\widetilde{S}_{\mu_1} \hat{\times} \widetilde{S}_{\mu_2}
\hat{\times} \cdots \hat{\times}\widetilde{ S}_{\mu_s}=\widetilde{S}_{\mu}.
\end{split}
 \end{equation}
Furthermore, if we set
$ \widetilde{\Gamma}_{\mu}:=\Gamma^n\ltimes T_{\mu}\simeq\Gamma^n \ltimes \widetilde{S}_{\mu},$ then it
 can be viewed as a subgroup of
$\widetilde{\Gamma}_n.$

We now use the Mackey-Wigner method of little groups (cf. \cite{Serre}) to construct the
irreducible spin characters of $\widetilde{\Gamma}_n.$
In the following we will let $\pi_{\nu}$ be an irreducible spin
 $\widetilde{S}_{\mu}$-module corresponding to $s$ strict partitions
 $\nu^1, \cdots, \nu^s$ such that $|\nu^i|=\mu_i$ and  $\chi_{\nu}$ be the spin character
 afforded by $\pi_{\nu}$. For abelian groups
we may simply use the same letter to denote a representation as well as its character.

Now fix $i:=(i_1, \cdots, i_{s})$, a combination from $\{0, \cdots, r\}$ and let
$\widehat{\pi}_{\nu}$ be the irreducible spin module of $\widetilde{\Gamma}_{\mu}$ obtained by composing $\pi_{\nu}$
with the canonical projection
$\widetilde{\Gamma}_{\mu}\longrightarrow \widetilde{S}_{\mu}$. Then $\gamma^{\mu}_i\otimes\widehat{\pi}_{\nu}$
is an irreducible spin module of $\widetilde{\Gamma}_{\mu}$.   Finally we
define
$$\label{theta}\Theta_{\mu, i}^{\nu}\triangleq 
\mbox{Ind}_{\widetilde{\Gamma}_{\mu}}^{\widetilde{\Gamma}_n}(\gamma_i^{\mu}\otimes
\chi_{\widehat{\pi}_{\nu}}). $$
By the Mackey-Wigner method it is clear that $\Theta_{\mu, i}^{\nu}$ is an irreducible spin character
and any irreducible spin character is of this form (also see \cite{FJW} for a direct argument).
We will simply write $\Theta_{\mu, i}^{\nu}$ by
$\Theta_{i}^{\nu}$, where $\nu\in\mathcal{SP}_n(\Gamma_*)$, as it is
an induced character from the Young subgroup $\widetilde{\Gamma}_{\mu}$ and $\mu_j=|\nu^j|, j=1, \cdots, r+1$.
Here and later we allow some $\nu^j$ to be empty, thus even if
$\nu\in\mathcal{SP}_n(\Gamma_*)$ we often write out only the non-empty
partitions, so the associated weight partition $\mu$ with
$\mu_j=|\nu^j|$ is a partition with length $l(\mu)\leq r+1$.

\subsection{The irreducible spin super character table of $\widetilde{\Gamma}_n$}

When $n<4$, the spin group $\widetilde{S}_n$ is a direct product of $\mathbb Z_2$ and $S_n$, so we will assume
$n\geq 4$ throughout this section.
Let $(g,\sigma)\in \Gamma_n$, where $\sigma$ has type $\rho=(\rho(c))_{c\in \Gamma_{*}}$.
The preimage elements are then $(g,z^pt_{\rho})$, $p=0, 1$ (see (\ref{eq:preimage})).
These two elements are representatives of the conjugacy classes $D_{\rho}^+$ and $D_{\rho}^-$ respectively.

\begin{prop}\label{prop1} Let $\nu=(\nu^1, \cdots, \nu^s)\in \mathcal{SP}_n(\Gamma_{*})$
with $|\nu^j|=\mu_j$ and let $i=(i_1, \cdots, i_s)\in I$,
then for $\rho=(\rho^1, \cdots, \rho^s)\in\mathcal{P}_n(\Gamma_{*})$
such that $|\rho^j|=\mu_j$, the character values of $\Theta_{i}^{\nu}$
 at the conjugacy classes $D_{\rho}^{\pm}$ are given by
\begin{equation}\Theta_{i}^{\nu}(D_{\rho}^{\pm})=\pm K_{\rho}\prod_{j=1}^s\big(\prod_{c\in
\Gamma_{*}}\gamma_{i_j}^{l(\rho^{j}(c))}\big)\cdot
\chi_{\nu}(t_{\rho}),
\end{equation}
 where $K_{\rho}$ is the number of left cosets $T$ of $\widetilde{\Gamma}_{\mu}$ in $\widetilde{\Gamma}_n$
  such that $(g,z^pt_{\rho})T=T.$
\end{prop}

\begin{proof} Since two elements of $\widetilde{\Gamma}_n$ are conjugate if and only if they have the same
type. So for each transversal $t$ of the left coset of $\widetilde{\Gamma}_{\mu}$ in $\widetilde{\Gamma}_n$,
 both $(g,z^pt_{\rho})$ and $t^{-1}(g,z^pt_{\rho})t$ have the same type $\rho$. Let $V_{i_j}(j=1,\cdots,s)$
  be a $\widetilde{\Gamma}_n$-module affording the character $\gamma_{i_j}\in \Gamma^{*}$. We will compute the
   character $\gamma_i^{\mu}\otimes \chi_{\widehat{\pi}_{\nu}}$ of the representation $V^{\otimes\mu_1}_{i_1}
   \otimes\cdots\otimes V^{\otimes\mu_s}_{i_s}\otimes W.$
If $\tilde{x}\in \widetilde{\Gamma}_m$ and $\tilde{y}\in \widetilde{\Gamma}_{n-m}$ then $\tilde{x}$ acts
on the first $m$ factors of $V^{\otimes\mu_1}_{i_1}\otimes\cdots\otimes V^{\otimes\mu_s}_{i_s}$ and $\tilde{y}$
on the last $n-m$ factors, it is clear that
\begin{equation}
\gamma_i^{\mu}\otimes \chi_{\widehat{\pi}_{\nu}}(\tilde{x}\hat{\times} \tilde{y})=\gamma_i^{\mu}
\otimes\chi_{\widehat{\pi}_{\nu}}(\tilde{x})\cdot\gamma_i^{\mu}\otimes\chi_{\widehat{\pi}_{\nu}}(\tilde{y}).
\end{equation}
Therefore it is enough to compute $\gamma_i^{\mu}\otimes \chi_{\widehat{\pi}_{\nu}}(g,z^pt_{\rho})$
when  $t_{\rho}=[1,\cdots,\mu_1]$ $\cdots[\mu_1+\cdots+\mu_{s-1}+1,\cdots,n]$
is an $(\mu_1,\cdots,\mu_s)$-cycle. For this purpose, let $e_{i_j}$ be a basis of $V_{i_j}$(as $\Gamma$
is a cyclic group) and let $ge_{i_j}=\gamma_{i_j}(g)e_{i_j}, \gamma_{i_j}(g)\in \mathbb{C}$.
As $t_{\rho}\cdot\gamma_i^{\mu}=\gamma_i^{\mu}$
for $t_{\rho}\in \widetilde{S}_{\mu}$, it follows that
\begin{equation}
\begin{split}
&(g,z^pt_{\rho})(e_{i_1}^{\otimes \mu_1}\otimes\cdots\otimes e_{i_s}^{\otimes\mu_s}\otimes w)\\
=&g_1(e_{i_1})\otimes\cdots\otimes g_{\mu_1}(e_{i_1})\otimes\cdots \\
&\ \ \ \ \ \ \cdots\otimes g_{\mu_1+\cdots+\mu_{s-1}+1}(e_{i_s})\otimes\cdots\otimes g_{\mu_1+\cdots+\mu_s}(e_{i_s})\otimes z^pt_{\rho}(w)\\
=&\gamma_{i_1}(g_{\mu_1}\cdots g_{1})\cdots \gamma_{i_s}( g_n \cdots g_{n-\mu_s})
(e_{i_1}^{\otimes \mu_1}\otimes\cdots\otimes e_{i_s}^{\otimes\mu_s}\otimes z^pt_{\rho}(w))\\
\end{split}
\end{equation}
If for each $j\in \{1,\cdots,s\}$, the cycle-product
$g_{\Sigma_{k=0}^{j}\mu_k}\cdot\dots\cdot
g_{\Sigma_{k=0}^{j-1}\mu_k+2}\cdot g_{\Sigma_{k=0}^{j-1}\mu_{k}+1}$
$(\mu_0=0)$ lies in $c\in \Gamma_{*}$, then we have
$$\gamma_i^{\mu}\otimes \chi_{\widehat{\pi}_{\nu}}
(g,z^pt_{\rho})=\prod_{j=1}^s\big(\prod_{c\in\Gamma_{*}}\gamma_{i_j}(c)^{l(\rho^j(c))}\big)\chi_{\nu}(z^pt_{\rho}).$$
Subsequently
\begin{equation} \label{Q1}
\begin{split}
\Theta_{i}^{\nu}(D_{\rho}^{\pm})=&\pm \sum_{\widetilde{x}\in\widetilde{\Gamma}_n}
\frac{1}{|\widetilde{\Gamma}_{\mu}|}\gamma_i^{\mu}\otimes\chi_{\widehat{\pi}_{\nu}}(\widetilde{x}^{-1}(g,t_{\rho})\widetilde{x})\\
=&\pm K_{\rho}\prod_{j=1}^s\big(\prod_{c\in\Gamma_{*}}\gamma_{i_j}(c)^{l(\rho^j(c))}\big)\chi_{\nu}(t_{\rho}).
\end{split}
\end{equation}
\end{proof}
For $j=1,\cdots,s$, let $\nu^j$ be a partition valued function on $\Gamma_{*}$ and denote
\begin{equation}
\begin{split}
\nu^j
=((\nu^j_{1},\cdots,\nu^j_{j_1}),
\cdots,(\nu^j_{j_1+\cdots+j_{k-1}+1},\cdots,
\nu^j_{{j_1+\cdots+j_{k-1}+j_k}}))
\end{split}
\end{equation}
in $ \mathcal{P}_{\mu_j}(\Gamma_{*})$, where $|\nu^j|=\mu_j$ and $j_1+j_2+\cdots+j_k=l(\nu^j)$. Then
$\nu=\nu^1\cup\cdots\cup\nu^s$ is a partition valued function in $\mathcal{P}_n(\Gamma_{*})$.
For each $c\in\Gamma_{*}$ and a partition $\lambda\in\mathcal P$, we
define the characteristic 
partition $c^{\lambda}\in
\mathcal{P}(\Gamma_{*})$ by
$$c^{\lambda}(c)=\lambda, \ \ c^{\lambda}(c^{'})={\emptyset }, ~\hbox{for}~ c^{'}\neq c.$$

Let $c^{(\tilde{\nu}^j)}:=c_{i_0}^{(\nu^j_{1})}\cup c_{i_1}^{(\nu^j_{2})}\cup\cdots\cup c_{i_{l(\nu^j)}}^{(\nu^j_{l(\nu^j)})}
$, then this union is a characteristic partition valued function in
 $\mathcal{P}_{\mu_j}(\Gamma_{*})$ supported only at $c_{i_0},c_{i_1},\cdots,c_{i_{l(\nu^j)}}$.
  Thus
   $$\tilde{\nu}:=c^{(\tilde{\nu}^1)}\cup\cdots\cup c^{(\tilde{\nu}^s)}$$
  is a  {\it characteristic
partition-valued function} in $\mathcal{P}_{n}(\Gamma_{*})$. Let
$\bar{\nu}=(\bar{\nu}^1,\cdots,\bar{\nu}^s)
=\bigcup_{j=1}^s(\bigcup_{c\in \Gamma_{*}} \nu^j(c))$, where
$\bar{\nu}^j=\bigcup_{c\in \Gamma_{*}} \nu^j(c)$ is a partition of
$\mu_j$. For $\nu, \xi\in \mathcal{P}_n(\Gamma_{*})$, we say they
are in the same class
 if $\bar{\nu}$ and $\bar{\xi}$ have the same partition parts,
and denote by $\bar{\nu}$ (or $\bar{\xi}$) the type of this class. We denote by $[\tilde{\nu}]$ the set
of characteristic partition valued functions with type $\bar{\nu}$. It is easy to see that the cardinality
of $[\tilde{\nu}]$ is $|\Gamma_{*}|^{l(\nu)}$.
For $(g,\sigma)\in \Gamma_n$, if its corresponding partition valued function is $\nu$,  then the cycle type of the permutation $\sigma$ is the type of class $[\tilde{\nu}]$.

For $\mu=(\mu_1,\cdots,\mu_s)\in\Omega$, let $\Delta_{\bar{\nu}^j}$
be an irreducible spin character of $\widetilde{S}_{\mu_j}$
corresponding to a strict partition $\bar{\nu}^j$ of $\mu_j$. Suppose that there are $k$ associate spin and $s-k$
double spin characters in $\{\Delta_{\bar{\nu}^1},\cdots,\Delta_{\bar{\nu}^s}\}$.
We denote by $\chi_{\nu}$ the character of the starred tensor product $\Delta_{\bar{\nu}^1}\circledast\cdots\circledast\Delta_{\bar{\nu}^s}$ when it is a double spin irreducible module.
When the irreducible component $\Delta_{\bar{\nu}^1}\circledast\cdots \circledast\Delta_{\bar{\nu}^s}$ is an
associate spin module, we choose $\chi_{\nu}^{+}$ to be the irreducible character
such that the basic spin character $\chi_{(\bar{\nu}^j_i)}$ 
  ( here each part $\bar{\nu}^j_i$ of the partition $\bar{\nu}^j$ corresponds a basic spin character $\chi_{(\bar{\nu}^j_i)}$)
appears with positive multiplicity in $Res(\chi_{\bar{\nu}^j}^{+})|_{\widetilde{S}_{\bar{\nu}^j_i}}$ 
for each $\bar{\nu}^j_i$.
Then the associated character is denoted by $\chi_{\nu}^{-}$.
Correspondingly the induced character $\Theta_{i}^{\nu}=\mbox{Ind}_{\widetilde{\Gamma}_{\mu}}^{\widetilde{\Gamma}_n}
 (\gamma_i^{\mu}\otimes\chi_{\nu})
 $
is a double spin character when $n-l(\nu)$ is even and $(\Theta_{i}^{\nu})^{\pm}=\mbox{Ind}_{\widetilde{\Gamma}_{\mu}}^{\widetilde{\Gamma}_n}
 (\gamma_i^{\mu}\otimes\chi_{\nu}^{\pm})$ are
associate spin characters when $n-l(\nu)$ is odd.

Let $t_{\nu}=t_{\nu^1}\cdots t_{\nu^s}\in \widetilde{S}_{\mu}$ such that $t_{\nu^i}\in \widetilde{S}_{\mu_i}$ for $i\in \{1,\cdots,s\}$. If each $\nu^i\in \mathcal{SP}^1_{\mu_i}(\Gamma_{*})$  and $s$ is odd, then we have (cf. (\ref{eqf}))
\begin{equation}\label{eq3}
\begin{split}
(\chi_{\nu})^{\pm}(t_{\nu})=&\pm\Delta_{\bar{\nu}^1}\circledast \cdots\circledast \Delta_{\bar{\nu}^s}(t_{\nu})\\
=&\pm(2\sqrt{-1})^{\frac{s-1}{2}}\Delta_{\bar{\nu}^1}(t_{\nu^1})\cdots\Delta_{\bar{\nu}^s}(t_{\nu^s})\\
=&\pm2^{\frac{s-1}{2}}\cdot(\sqrt{-1})^{\frac{n-l(\nu)+2s-1}{2}}\sqrt{\frac{\la_1\cdots\la_{l}}{2^s}}\\
=&\pm(\sqrt{-1})^{\frac{n-l(\nu)+2s-1}{2}}\sqrt{\frac{\la_1\cdots\la_{l}}{2}},
\end{split}
\end{equation}
where $(\la_1,\cdots,\la_l)$ is the type of $\nu$. Hence
\begin{equation}\label{eq4}
\begin{split}
(\chi_{\nu})^{\pm}(t_{\nu})\overline{(\chi_{\nu})^{\pm}(t_{\nu})} =
\frac{\la_1\cdots \la_{l}}{2}.
\end{split}
\end{equation}

\begin{prop}\label{p3} Let $\mu=(\mu_1,\cdots,\mu_s)\in \Omega$, $\la=(\la_1,\cdots,\la_l)$ be the type of $\nu=\nu^1\cup\cdots\cup\nu^s$
and $i\in I$.
If $s$ is odd and each $\nu^j$ is in $\mathcal{SP}_{\mu_j}^1(\Gamma_{*})$, then there are two associate
 irreducible spin characters $(\Theta_{i}^{\nu})^{\pm}$
of $\widetilde{\Gamma}_n$. For $\rho\in \mathcal{SP}^1_n(\Gamma_{*})$, the characters $(\Theta_{i}^{\nu})^{\pm}$ are given according to

(i) when $\rho=\rho^1\cup\cdots\cup\rho^s\in [\tilde{\nu}]$, then
$$(\Theta_{i}^{\nu})^{\pm}(D^{+}_{\rho})=\pm K_{\rho}
\prod_{j=1}^s(\prod_{c\in\Gamma_{*}}\gamma_{i_j}(c)^{l(\rho^j(c))})(\sqrt{-1})^{\frac{n-l(\la)+2s-1}{2}}
\sqrt{\frac{\la_1\cdots \la_{l}}{2}},$$
 where $K_{\rho}$ is the
number of left cosets $T$ of $\widetilde{\Gamma}_{\mu}$ in
$\widetilde{\Gamma}_n$ such that $(g,t_{\rho})T=T$.

(ii) when $\rho\notin [\tilde{\nu}]$, one has $(\Theta_{i}^{\nu})^{\pm}({\rho})=0.$
\end{prop}
\begin{proof} The first assertion follows from Proposition
\ref{prop1} and equation (\ref{eq3}). As for the second statement,
by the standard inner product we have
\begin{align}\label{1}\nonumber
&\langle (\Theta_{i}^{\nu})^{\pm},(\Theta_{i}^{\nu})^{\pm}\rangle \\
=&(\sum_{\rho\in \mathcal{OP}_n(\Gamma_{*})}
+\sum_{\rho\in \mathcal{SP}^1_n(\Gamma_{*})})\frac{1}{\widetilde{Z}_{\rho}}(\Theta_{i}^{\nu})^{\pm}(\rho)\overline{(\Theta_{i}^{\nu})^{\pm}(\rho)}.
\end{align}
 Recall that $(\Theta_{i}^{\nu})^{+}\oplus(\Theta_{i}^{\nu})^{-}$ can be regarded as the character of an irreducible supermodule of type $Q$.
 Hence by the inner product of super characters and the fact that $(\Theta_{i}^{\nu})^{+}(\rho)=(\Theta_{i}^{\nu})^{-}(\rho)$ when $\rho\in \mathcal{OP}_n(\Gamma_{*})$. It is easy to see that
\begin{equation}\label{3}
\sum_{\rho\in
\mathcal{OP}_n(\Gamma_{*})}\frac{1}{\widetilde{Z}_{\rho}}
(\Theta_{i}^{\nu})^{\pm}(\rho)\overline{(\Theta_{i}^{\nu})^{\pm}(\rho)}=1/2,
\end{equation}
so the second summand in (\ref{1}) should
also be equal to $\frac{1}{2}$. On the other hand, by equation (\ref{Inner}) it follows that
\begin{equation}\label{2}
\begin{split}
& \sum_{\rho\in \mathcal{SP}^1_n(\Gamma_{*})}
\frac{1}{\widetilde{Z}_{\rho}}(\Theta_{i}^{\nu})^{\pm}(\rho)\overline{(\Theta_{i}^{\nu})^{\pm}(\rho)}\\
=&\sum_{\rho\in
\mathcal{SP}^1_n(\Gamma_{*})}\frac{1}{\widetilde{Z}_{\rho}}
(\gamma_i^{\mu}\otimes\chi_{\widehat{\pi}_{\nu}}^{\pm})(\rho)
\overline{(\gamma_i^{\mu}\otimes\chi_{\widehat{\pi}_{\nu}}^{\pm})(\rho)} \ \ (\hbox{by}~ (\ref{Inner}))\\
  \geq&\sum_{\rho\in [\tilde{\nu}]}\frac{2}{\widetilde{Z}_{\rho}}
(\gamma_i^{\mu}\otimes\chi_{\widehat{\pi}_{\nu}}^{\pm})(D^{+}_{\rho})
\overline{(\gamma_i^{\mu}\otimes\chi_{\widehat{\pi}_{\nu}}^{\pm})(D^{+}_{\rho})}
\ \ (\hbox{as}~ D_{\rho}^{-}=zD_{\rho}^{+})\\
\geq&\sum_{\rho\in[\tilde{\nu}]}\frac{\chi_{\nu}^{\pm}(t_{\rho})\overline{\chi_{\nu}^{\pm}(t_{\rho})}}{\prod_{j=1}^s(\prod_{c\in\Gamma_{*}}z_{\rho^j(c)}\zeta_c^{l(\rho^j(c))})}
 \ \
(\hbox{as}~ |\gamma_{i_j}(c)|^2=1)\\
\geq&\sum_{\rho\in[\tilde{\nu}]}\frac{1}{\la_1\cdots\la_l\cdot(r+1)^{l(\rho)}}
\cdot\frac{\la_1\cdots\la_l}{2} \ \ (\hbox{by} ~ (\ref{eq4}) )  \\
\geq&\frac{1}{2}\ \ (\hbox{as}~ |[\widetilde{\nu}]|=(r+1)^{l(\rho)}).
\end{split}
\end{equation}
Combining (\ref{1}), (\ref{3}) and (\ref{2}) we have that
\begin{align*}
\frac12&=\sum_{\rho\in \mathcal{SP}^1_n(\Gamma_{*})}\frac{1}{\widetilde{Z}_{\rho}}(\Theta_{i}^{\nu})^{\pm}(\rho)
\overline{(\Theta_{i}^{\nu})^{\pm}(\rho)}\\
&\geq \sum_{\rho\in [\tilde{\nu}]}\frac{1}{\widetilde{Z}_{\rho}}(\Theta_{i}^{\nu})^{\pm}(\rho)\overline{(\Theta_{i}^{\nu})^{\pm}(\rho)}\geq\frac12
\end{align*}
which forces $(\Theta_{i}^{\nu})^{\pm}({\rho})=0$ if
$\rho\notin [\tilde{\nu}]$.
\end{proof}

\begin{example} Consider $\widetilde{\Gamma}_{13}$ with $\Gamma=\langle a|a^3=1\rangle$. Let $\Gamma^{*}=\{\gamma_0,\gamma_1,\gamma_2\}$ and
$\Gamma_{*}=\{c^0,c^1,c^2\}$, where $\gamma_i(c^j)=w^{ij}$, and $w=-\frac12+\frac{\sqrt{-3}}2$.
The irreducible characters of $\Gamma^{13}$ are classified into $|\mathcal{P}_{13}|$ orbits
under the action of $S_{13}$. Here $\mathcal{P}_{13}$ is the set of partitions of $13$.
We list some of these class orbits and compute the associated spin characters as follows.

The first class orbit is $\Phi_{(13)}=\{\mathcal{O}(\gamma_{i}^{\otimes 13})=\{\gamma_{i}^{\otimes 13}\}|i=0,1,2\},$ then
$$T_{(13)}=\{z^pt_{\rho}\in \widetilde{S}_{13}|z^pt_{\rho}\cdot \gamma_{i}^{\otimes 13}=\gamma_{i}^{\otimes 13}\}\simeq \widetilde{S}_{13}
,\ \  \widetilde{\Gamma}_{(13)}=\Gamma^{13}\ltimes T_{(13)}=\widetilde{\Gamma}_{13}.$$
For $i=1, \nu=((5,4,3,1)_c)_{c\in\Gamma_{*}}\in \mathcal{SP}^1_{13}({\Gamma_{*}})$ and
$\rho=((54)_{c^0},(31)_{c^2})\in[\widetilde{\nu}]$
 (i.e. there are one 5-cycle and one 4-cycle such that their cycle-products lie in $c^0$,
 the same is true for $(31)_{c^2}$), the type of the class $[\tilde{\nu}]$ is $\la=(5,4,3,1)$. Then
 \begin{equation}
\begin{split}
&(\gamma_{1}^{\otimes 13}\otimes\Delta_{\bar{\nu}}^{\pm})(D_{\rho}^{+})\\
 =& \gamma_1(c^0)^{l(\rho(c^0))}\cdot\gamma_1(c^2)^{l(\rho(c^2))}\cdot\Delta_{\bar{\nu}}^{\pm}(t_{\rho})\\
 =&\pm1\cdot (w^2)^2\cdot (\sqrt{-1})^{\frac{13-4+2-1}{2}}\sqrt{\frac{5\times4\times3\times1}{2}}\\
 =& \pm\sqrt{-30}w.
 \end{split}
 \end{equation}
 and $(\gamma_{1}^{\otimes 13}\otimes\Delta_{\bar{\nu}}^{\pm})(\rho)=0$ if $\rho\notin[\widetilde{\nu}].$ As
\begin{equation}
\begin{split}
&\sum_{\rho\in[\tilde{\nu}]}\frac{2}{\widetilde{Z}_{\rho}}(\gamma_{1}^{\otimes 13}
\otimes\Delta_{\bar{\nu}}^{\pm})(D_{\rho}^{+})\overline{(\gamma_{1}^{\otimes
13}\otimes\Delta_{\bar{\nu}}^{\pm})
(D_{\rho}^{+})}\\
=& \sum_{\rho\in[\tilde{\nu}]}\frac{2}{2Z_{\rho}}\cdot \frac{(\pm\sqrt{-30}w)\cdot(\pm\overline{\sqrt{-30}w})}{z_{(5,4)}\zeta_{c^0}^2z_{(3,1)}\zeta_{c^2}^2}\\
=&|\Gamma_{*}|^4\cdot\frac{30}{5\cdot4\cdot3^2\cdot3\cdot1\cdot 3^2}\ \ (\hbox{as} \ \ \zeta_c=|\Gamma|=3)\\
=&\frac{1}{2}.
\end{split}
\end{equation}

The second class orbit is
$\Phi_{(5,4,4)}=\{\mathcal{O}(\gamma_{i}^{\otimes5}\otimes
\gamma_{j}^{\otimes 4}\otimes\gamma_{k}^{\otimes 4})|i,j,k=0,1,2\},$
$$T_{(5,4,4)}=\{z^pt_{\rho}\in \widetilde{S}_{13}|z^pt_{\rho}\cdot \gamma_{i}^{\otimes5}
\otimes\gamma_{j}^{\otimes 4}\otimes\gamma_{k}^{\otimes 4}=\gamma_{i}^{\otimes5}\otimes\gamma_{j}^{\otimes 4}
\otimes\gamma_{k}^{\otimes 4}\}\simeq \widetilde{S}_5\hat{\times}\widetilde{S}_4\hat{\times}\widetilde{S}_4,$$
$$\widetilde{\Gamma}_{(5,4,4)}=\Gamma^{13}\ltimes (\widetilde{S}_5\hat{\times}
\widetilde{S}_4\hat{\times}\widetilde{S}_4)
=\widetilde{\Gamma}_5\hat{\times}\widetilde{\Gamma}_4\hat{\times}\widetilde{\Gamma}_4.$$
For $i=2, j=1,k=0$ and $\bar{\nu}=((3,2),(4),(4))$, let
$\chi_{\nu}^{\pm}:=\pm\Delta_{\bar{\nu}^1}\circledast
\Delta_{\bar{\nu}^2}\circledast\Delta_{\bar{\nu}^3}$. For
$\rho=\rho^1\cup\rho^2\cup\rho^3=((3)_{c^1},(2)_{c^2})\cup((4)_{c^2})\cup((4)_{c^1})\in
\mathcal{SP}^1_{13}(\Gamma_{*})$, then
\begin{equation}
\begin{split}
&\gamma_{2}^{\otimes 5}\otimes \gamma_{1}^{\otimes 4}\otimes\gamma_{0}^{\otimes 4}
\otimes\chi_{\nu}^{\pm}(D_\rho^{+})\\
=&\gamma_2(c^1)^{l(\rho^1(c^1))}\cdot\gamma_2(c^2)^{l(\rho^1(c^2))}\cdot\gamma_1(c^2)^{l(\rho^2(c^2))}\cdot\gamma_0(c^1)^{l(\rho^3(c^1))}\chi_{\nu}^{\pm}(D_\rho^{+})\\
=&\pm w^2\cdot w^4\cdot w^2\cdot w^0\cdot(\sqrt{-1})^{\frac{13-4+6-1}{2}}\sqrt{\frac{3\times2\times4\times4}{2}}\\
=&\mp4\sqrt{-3}w^2,
\end{split}
\end{equation}
and $\gamma_{2}^{\otimes 5}\otimes \gamma_{1}^{\otimes 4}\otimes \gamma_{0}^{\otimes 4}
\otimes\chi_{\nu}^{\pm}\uparrow_{\widetilde{\Gamma}_{(5,4,4)}}^{\widetilde{\Gamma}_n}(\rho)=0$
 if $\rho\notin[\widetilde{\nu}]$. We see that

\begin{equation}\nonumber
\begin{split}
& \sum_{\rho\in[\tilde{\nu}]}\frac{2}{\widetilde{Z}_{\rho}}\gamma_{2}^{\otimes 5}\otimes \gamma_{1}^{\otimes 4}\otimes \gamma_{0}^{\otimes 4}\otimes\chi_{\nu}^{\pm}(D_{\rho}^{+})\overline{\gamma_{2}^{\otimes 5}\otimes \gamma_{1}^{\otimes 4}\otimes \gamma_{0}^{\otimes 4}\otimes\chi_{\nu}^{\pm}(D_{\rho}^{+})}\\
=&|\Gamma_{*}|^4\cdot\frac{1}{3\cdot2\cdot4\cdot4\cdot3^4}\cdot
|\mp4\sqrt{-3}w^2|^2\\
=&\frac{1}{2},
\end{split}
\end{equation}
where the type of the class $[\tilde{\nu}]$ is $\la=(4,4,3,2)$.
\end{example}

\begin{prop}
Let $\mu=(\mu_1,\cdots,\mu_s)\in \Omega$, $i\in I$, and $\nu=\nu^1\cup\cdots\cup\nu^s\in\mathcal{SP}^0_n(\Gamma_{*})$ such that
each $\nu^i\in \mathcal{SP}^0_{\mu_i}(\Gamma_{*})$. Then there is an irreducible  double spin character $\Theta_{i}^{\nu}$ of $\widetilde{\Gamma}_n$. Moreover,

(i) when $\rho=\rho^1\cup\cdots\cup\rho^s\in \mathcal{ OP}_n(\Gamma_{*})$ and $||\rho^i||=\mu_i$,
$$\Theta_{i}^{\nu}(D^{+}_{\rho})= 
\prod_{j=1}^s(\prod_{c\in\Gamma_{*}}\gamma_{i_j}(c)^{l(\rho^j(c))})\Delta_{\bar{\nu}^j}(t_{\rho^j})K_{\rho},$$
the values $\Theta_{i}^{\nu}(D_{\rho}^{+})$ are determined by the wreath product
Schur Q-functions (see \cite{FJW}).

(ii) otherwise, one has $\Theta_{i}^{\nu}(\rho)=0$.
\end{prop}

\begin{proof}
(i) The expression of $\Theta_{i}^{\nu}(D_{\rho}^{+})$ follows from Proposition \ref{prop1} and equation (\ref{eqf})
(the case of $k=0$). Moreover, if $\rho^i\in \mathcal{OP}_{\mu_i}(\Gamma_{*})$ then $\Delta_{\bar{\nu}^i}(t_{\rho^i})$ is determined by Schur Q-functions as in $(\ref{Q})$, hence the values  $\Theta_{i}^{\nu}(D_{\rho}^{+})$ are determined by wreath products of
Schur Q-functions.

(ii) If $\rho$ can not be decomposed as $\rho^1\cup\cdots\cup\rho^s$ such that $||\rho^i||=\mu_i$,
 then by the theory of induced characters it is easy to see that $\Theta_{i}^{\nu}({\rho})=0$. As $\nu\in \mathcal{SP}^0_n(\Gamma_{*})$, the corresponding
irreducible character $\Theta_{i}^{\nu}$ is a double spin character. Thus it can be regarded as an
irreducible super character
of type $M$.  Hence
$$<\Theta_{i}^{\nu},\Theta_{i}^{\nu}>_{\widetilde{\Gamma}_n}=1
=<\Theta_{i}^{\nu},\Theta_{i}^{\nu}>_{\mathcal{OP}_n(\Gamma_{*})}$$
which forces $\Theta_{i}^{\nu}({\rho})=0$ for
all $\rho=\rho^1\cup\cdots\cup\rho^s\in \mathcal{SP}_n^1(\Gamma_{*})$.
\end{proof}

For $\nu=\nu^1\cup\cdots\cup\nu^s\in \mathcal{SP}_n(\Gamma_{*})$,
let $J=\{j_1,\cdots,j_{k}\}$  be a maximal proper subset of
$\{1,\cdots,s\}$ such that $\nu^{i}$ is in
$\mathcal{SP}^1_{\mu_{i}}(\Gamma_{*})$ for $i\in
\{j_1,\cdots,j_{k}\}$. Let $J^{'}$ be the complement of $J$ in
$\{1,\cdots,s\}$. For $\rho=\rho^1\cup\cdots\cup\rho^s\in
\mathcal{SP}_n^1(\Gamma_{*})$, one sees that if
$\Delta_{\bar{\nu}^i}(t_{\rho^i})\neq 0$ then $\rho^i$
must be in $[\tilde{\nu}^i]$ for $i\in J$, and $\rho^i$ must be in
$\mathcal{OSP}_{\mu_i}(\Gamma_{*}):=\mathcal{OP}_{\mu_i}(\Gamma_{*})\cap
\mathcal{SP}_{\mu_i}(\Gamma_{*})$ for $i\in J^{'}$. So we have the
following results.

\begin{theorem}
Let $\nu=\nu^1\cup\cdots\cup\nu^s\in \mathcal{SP}^1_n(\Gamma_{*})$ and $i\in I$. Let
$\mu=(\mu_1,\cdots,\mu_s)\in \Omega$ be the weight partition
with $\mu_i=|\nu^i|$.
The character values of $(\Theta_{i}^{\nu})^{\pm}$ are computed as follows.
(i) When $\rho=\rho^1\cup\cdots\cup\rho^s \in\mathcal{SP}^1_n(\Gamma_{*})$ satisfies $\rho^i\in [\tilde{\nu}^i]$ for $i\in J$ and $\rho^i\in\mathcal{OSP}_{\mu_i}(\Gamma_{*})$ for $i\in J^{'}$, then
\begin{equation}\nonumber
\begin{split}(\Theta_{i}^{\nu})^{\pm}(D^{+}_{\rho})
=&\pm (\sqrt{-1})^{\frac{\sum_{j\in J}(\mu_j-l(\nu^j))+2|J|-1}{2}}\sqrt{\frac{\prod_{j\in
J}(\prod_{c\in\Gamma_{*}}\nu^j(c))}{2}}\cdot\\
&\prod_{j=1}^s(\prod_{c\in\Gamma_{*}}\gamma_{i_j}(c)^{l(\rho^j(c))})\cdot\prod_{j\in
J^{'}}\Delta_{\bar{\nu}^j}(t_{\rho^j})\cdot K_{\rho},
\end{split}
\end{equation}
where $K_{\rho}$ is the number of left cosets $T$ of
$\widetilde{\Gamma}_{\mu}$ in $\widetilde{\Gamma}_n$ such that
$(g,t_{\rho})T=T$, and the value of $\prod_{j\in
J^{'}}\Delta_{\bar{\nu}^j}(t_{\rho^j})$ is determined by the wreath
products of Schur Q-functions (see \cite{FJW}).

(ii) $(\Theta_{i}^{\nu})^{\pm}(\rho)=0$, otherwise.
\end{theorem}

\begin{proof}
The first assertion follows from
(\ref{Q1}), (\ref{eq3}) and Proposition \ref{p3}.

Now we consider the second part. For
$\nu=\nu^1\cup\cdots\cup\nu^s\in \mathcal{SP}^1_n(\Gamma_{*})$, let us
assume that
$\Delta_{\bar{\nu}^1},\cdots,\Delta_{\bar{\nu}^{2m-1}}$ are associate spin and $\Delta_{\bar{\nu}^{2m}},\cdots,\Delta_{\bar{\nu}^s}$ are double spin. Then  by (\ref{eq3}) and the second equation
of (\ref{2}), we have
\begin{equation}\label{eq:2}
\begin{split}
& \sum_{\rho\in \mathcal{SP}^1_n(\Gamma_{*})}\frac{1}{\widetilde{Z}_{\rho}}(\Theta_{i}^{\nu})^{\pm}(\rho)\overline{(\Theta_{i}^{\nu})^{\pm}(\rho)}\\
=&\sum_{\rho\in
\mathcal{SP}_n^1(\Gamma_{*})}\frac{2^{2(m-1)}}{\widetilde{Z}_{\rho}}
\big|\prod\limits_{j=1}^s(\prod\limits_{c\in\Gamma_{*}}\gamma_{i_j}(c)^{l(\rho^j(c))})\Delta_{\bar{\nu}^j}(\rho^j)\big|^2\\
=&\sum_{\rho=\rho^1\cup\cdots \cup\rho^s\in
\mathcal{SP}_n^1(\Gamma_{*})}
 \frac{2^{2(m-1)}}{\widetilde{Z}_{\rho}}
\big|\prod\limits_{j=1}^s\Delta_{\bar{\nu}^j}(\rho^j)\big|^2  \ \ (\hbox{as} \ \  |\gamma_{i_j}(c)|^2=1)\\
\geq&\sum_{\rho^j\in[\tilde{\nu}^j]:j\in J;\rho^j\in
\mathcal{SP}_{\mu_j}(\Gamma_{*}):j\in J^{'}}2^{2m-3}
(\prod\limits_{j=1}^{s}\frac{1}{Z_{\rho^j}}\big|\Delta_{\bar{\nu}^j}(\rho^j)\big|^2)
\end{split}
\end{equation}
For each $j\in J$,
$\frac{1}{Z_{\rho^j}}|\Delta_{\bar{\nu}^j}(\rho^j)|^{2}=\frac{|\Delta_{\bar{\nu}^j}(\bar{\rho}^j)|^2}{z_{\bar{\rho}^j}
(r+1)^{l(\bar{\rho}^j)}}$ just depends on the type of $\rho^j$. Therefore, they have the same value for any $\rho^j\in [\widetilde{\nu}^j]$. Because they
have the same type $\bar{\nu}^j$,
$\prod_{j=1}^{2m-1}\frac{1}{Z_{\rho^j}}|\Delta_{\bar{\nu}^j}(\rho^j)|^{2}$
is a constant for different $\rho^1\cup\cdots\cup \rho^{2m-1}\in
[\tilde{\nu}^1]\cup\cdots\cup
[\tilde{\nu}^{2m-1}]$, then the last expression in
(\ref{eq:2}) satisfies that
\begin{equation}
\begin{split}
\geq2^{2m-3}(\sum_{\rho^j\in[\tilde{\nu}^j]:j\in
J}\prod\limits_{j=1}^{2m-1}\frac{\big|\Delta_{\bar{\nu}^j}(\rho^j)\big|^2}{Z_{\rho^j}})
(\sum_{\rho^j\in \mathcal{OSP}_{\mu_j}(\Gamma_{*}):j\in
J^{'}}\prod\limits_{j=2m}^{s}\frac{\big|\Delta_{\bar{\nu}^j}(\rho^j)\big|^2}{Z_{\rho^j}}
)\\
\end{split}
\end{equation}
In the above we have used
 $\mathcal{OSP}_{\mu_j}(\Gamma_{*}):=\mathcal{OP}_{\mu_j}(\Gamma_{*})\cap\mathcal{SP}_{\mu_j}(\Gamma_{*})$, and then
\begin{equation}
\begin{split}
&\sum_{\rho^j\in \mathcal{OSP}_{\mu_j}(\Gamma_{*}):j\in
J^{'}}\prod\limits_{j=2m}^{s}\frac{\big|\Delta_{\bar{\nu}^j}(\rho^j)\big|^2}{Z_{\rho^j}}\\
=&\sum_{\rho^j\in
\mathcal{OSP}_{\mu_j}(\Gamma_{*}):j\in J^{'}}\prod\limits_{j=2m}^s
\frac{|\Delta_{\bar{\nu}^{j}}(D_{\rho^{j}})|^2}
{\prod\limits_{c\in\Gamma_{*}}z_{\rho^j(c)}\zeta_{c}^{l(\rho^j(c))}}\\
=&\sum_{\bar{\rho}^j\in \mathcal{OP}_{\mu_j}}
\prod_{j=2m}^{s}\frac{|\Gamma_{*}|^{l(\rho^j)}|\Delta_{\bar{\nu}^j}(z^pt_{\rho^j})|^2}{z_{\bar{\rho}^j}(r+1)^{l(\rho^j)}}\\
=&\prod_{j=2m}^{s}\big(\sum_{\bar{\rho}^j\in \mathcal{OP}_{\mu_j}}
\frac{|\Delta_{\bar{\nu}^j}(z^pt_{\rho^j})|^2}{z_{\bar{\rho}^j}}\big)\\
=&\prod_{j=2m}^s<\Delta_{\bar{\nu}^j},\Delta_{\bar{\nu}^j}>_{\widetilde{S}_{\mu_j}} \ \ (\Delta_{\bar{\nu}^j} \hbox{~is~double})\\
=&1.\\
\end{split}
\end{equation}
Subsequently, equation (\ref{eq:2})
\begin{equation}
\begin{split}
&\geq2^{2m-3}(\sum_{\rho^j\in[\tilde{\nu}^j]:j\in
J}\prod\limits_{j=1}^{2m-1}\frac{\big|\Delta_{\bar{\nu}^j}(\rho^j)\big|^2}{Z_{\rho^j}})\\
&\geq\sum_{\rho^j\in [\tilde{\nu}^j]:j\in J}
\frac{2^{2(m-1)}\prod\limits_{j=1}^{2m-1}|\Delta_{\bar{\nu}^{j}}(D^{+}_{\rho^{j}})|^2}
{\prod\limits_{j=1}^{2m-1}(\prod_{c\in\Gamma_{*}}z_{\rho^j(c)}\zeta_c^{l(\rho^j(c))})} \ \ (D_{\rho^j}^{-}=zD_{\rho^j}^{+})\\
&\geq2^{2(m-1)}\prod\limits_{j=1}^{2m-1}\frac{|\Gamma_{*}|^{l(\rho^j)}\big|\sqrt{{\prod_{c\in\Gamma_{*}}\nu^{j}(c)}/{2}}\big|^2}
{\prod_{c\in\Gamma_{*}}\nu^j(c)(r+1)^{l(\rho^j(c))}}\\
&\geq\frac{1}{2}.
\end{split}
\end{equation}
We have pointed out that
$
\sum_{\rho\in
\mathcal{SP}^1_n(\Gamma_{*})}\frac{1}{\widetilde{Z}_{\rho}}
(\Theta_{i}^{\nu})^{\pm}(\rho)\overline{(\Theta_{i}^{\nu})^{\pm}(\rho)}=1/2,
$
which forces $(\Theta_{i}^{\nu})^{\pm}(\rho)=0$ if $\rho^i\notin
[\tilde{\nu}^i]$ for $i\in J$ or
$\rho^i\notin\mathcal{OSP}_{\mu_i}(\Gamma_{*})$ for $i\in J^{'}$. 
\end{proof}

\begin{cor}
 For $\mu\in\Omega$, $i\in I$ and let $\nu=((\mu_1),(\mu_2),\cdots,(\mu_s))\in \mathcal{SP}^1_n(\Gamma_{*})$. Suppose all $\mu_j$
are even integers,
then for
$\rho=\rho^1\cup\cdots\cup\rho^s\in \mathcal{SP}^1_n{(\Gamma_{*})}$, the values of the irreducible spin characters $(\Theta_{i}^{\nu})^{\pm}$ at the conjugacy classes
$D_{\rho}^{+}$ are given by
\begin{equation}\nonumber
(\Theta_{i}^{\nu})^{\pm}(D_{\rho}^{+})= \left\{\begin{array}{lc}
\pm(\sqrt{-1})^{\frac{n+s-1}{2}}\sqrt{\frac{\mu_1\cdots\mu_s}{2}}
\prod\limits_{k=1}^s\gamma_{i_k}(c^{j_k})K_{\rho},&\rho^i=c_{j_i}^{(\mu_i)}, \\
0,&otherwise,
\end{array}
\right.
\end{equation}
where $K_{\rho}$ is the number of left cosets $T$ of
$\widetilde{\Gamma}_{\mu}$ in $\widetilde{\Gamma}_n$ such that
$(g,t_{\rho})T=T$.
\end{cor}
\begin{proof} We just need to check
\begin{equation}
\begin{split}(\Theta_{i}^{\nu})^{\pm}(D_{\rho}^{+})&=
\pm(\sqrt{-1})^{\frac{n}{2}}2^{\frac{s-1}{2}}\sqrt{\frac{\mu_1}{2}}\cdots\sqrt{\frac{\mu_s}{2}}
\prod\limits_{k=1}^s\gamma_{i_k}(c^{j_k})K_{\rho}\\
&=\pm(\sqrt{-1})^{\frac{n}{2}}\sqrt{\frac{\mu_1\cdots\mu_s}{2}}
\prod\limits_{k=1}^s\gamma_{i_k}(c^{j_k})K_{\rho}
\end{split}
\end{equation} 
\end{proof}

In particular, when $\nu=(n)$ (i.e. $s=1$) and $\rho\in \mathcal{SP}^1_n(\Gamma_{*})$ (also see Corollary 4.5 in \cite{FJW})
\begin{equation}\nonumber
(\Theta_{i}^{\nu})^{\pm}(D_{\rho}^{+})= \left\{\begin{array}{lc}
\pm(\sqrt{-1})^{\frac{n}{2}}\sqrt{\frac{n}{2}}\gamma_i(c), &\rho=c^{(n)}, \\
0,&otherwise.
\end{array}
\right.
\end{equation}

\medskip

{\small\noindent{\bf Acknowledgments}
The second named author
gratefully acknowledges the partial support of the Max-Planck Institute for Mathematics in Bonn,
Simons Foundation 
and NSFC 
during this work.
}

\bibliographystyle{spmpsci}      


\end{document}